\documentclass[11pt,twoside]{amsart}
\usepackage{amsxtra}
\usepackage{color}
\usepackage{amsopn}
\usepackage{amsmath,amsthm,amssymb}
\usepackage{mathrsfs,mathtools}
\usepackage{enumerate}
\usepackage{xcolor}

\newtheorem{theorem}{Theorem}[section]

\newtheorem{proposition}[theorem]{Proposition}
\newtheorem{lemma}[theorem]{Lemma}
\newtheorem*{theorem*}{Theorem}
\theoremstyle{definition}

\newtheorem{remark}[theorem]{Remark}

\newcommand{\R}{{\mathbb R}}

\newcommand{\C}{{\mathbb C}}

\newcommand{\bC}{\mathbb{C}}

\newcommand{\bR}{\mathbb{R}}

\newcommand{\beq}{\begin{equation}}
\newcommand{\eeq}{\end{equation}}
\renewcommand{\a}{\alpha}
\renewcommand{\b}{\beta}

\renewcommand{\d}{\delta}

\newcommand{\g}{\gamma}

\renewcommand{\l}{\lambda}

\newcommand{\s}{\sigma}

\renewcommand{\L}{\Lambda}

\newcommand{\U}{{\mathrm U}}

\newcommand{\n}{\nabla}

\DeclareMathOperator\End{End}

\DeclareMathOperator\Ad{Ad}
\DeclareMathOperator\ad{ad}
\DeclareMathOperator\vol{vol}

\DeclareMathOperator{\Span}{Span}

\newcommand{\ga}{\mathfrak{a}}
\newcommand{\gb}{\mathfrak{b}}

\newcommand{\gf}{\mathfrak{f}}
\renewcommand{\gg}{\mathfrak{g}}

\newcommand{\gk}{\mathfrak{k}}
\newcommand{\gl}{\mathfrak{l}}
\newcommand{\gm}{\mathfrak{m}}
\newcommand{\gn}{\mathfrak{n}}

\newcommand{\gq}{\mathfrak{q}}

\newcommand{\gs}{\mathfrak{s}}

\newcommand{\gu}{\mathfrak{u}}

\newcommand{\gz}{\mathfrak{z}}

\newcommand{\so}{\mathfrak{so}}

\def\sideremark#1{\ifvmode\leavevmode\fi\vadjust{
		\vbox to0pt{\hbox to 0pt{\hskip\hsize\hskip1em
				\vbox{\hsize3cm\tiny\raggedright\pretolerance10000
					\noindent #1\hfill}\hss}\vbox to8pt{\vfil}\vss}}}%
\numberwithin{equation}{section}

\title[Hermitian Curvature Flow on compact homogeneous spaces]{Hermitian Curvature Flow on compact homogeneous spaces}

\author{Francesco Panelli and Fabio Podest\`a }
\address{ Dipartimento di Matematica e Informatica "Ulisse Dini", Universit\`a di Firenze, V.le Morgagni 67/A, 50100 Firenze, Italy}
\email{fabio.podesta@unifi.it,\ francesco.panelli@unifi.it}

\date{\today}


\begin{document}
	\begin{abstract} We study a version of the Hermitian curvature flow on compact homogeneous complex manifolds. We prove that the solution has a finite exstinction time $T>0$
		and we analyze its behaviour when $t\to T$. We also determine the invariant static metrics and we study the convergence of the normalized flow to one of them.  
		
	\end{abstract}
	
\maketitle
\section{Introduction}
Given a Hermitian manifold $(M,J,h)$ it is well-known that there exists a family of metric connections leaving the complex structure $J$ parallel (see \cite{G}). Among these the Chern connection is particularly interesting and provides different Ricci tensors which can be used to define several meaningful parabolic metric flows preserving the Hermitian condition and generalizing the classical Ricci flow in the non-K\"ahler setting. In \cite{Gi} Gill introduced an Hermitian flow on a compact complex manifold involving the first Chern-Ricci tensor, namely the one whose associated 2-form represents the first Chern class of $M$ (see also \cite{TW} for further related results). In \cite{ST} Streets and Tian introduced a family of Hermitian curvature flows (HCFs) involving the second Ricci tensor $S$ together with an arbitrary symmetric Hermitian form $Q(T)$ which is quadratic in the torsion $T$ of the Chern connection:
$$
h_t'=-S+Q(T).
$$  
For any admissible $Q(T)$ the corresponding flow is strongly parabolic and the short time existence of the solution is estabilished. This family includes geometrically interesting flows as for instance the {\em pluriclosed flow} that was previously introduced in \cite{ST2} and preserves the pluriclosed condition $\partial\overline{\partial}\omega=0$. More recently Ustinovskiy  focused on a particular choice of $Q(T)$ obtaining another remarkable flow, which we will call $\text{\it HCF}_{\text{\it U}}$ for brevity, with several geometrically relevant features (see \cite{U}). In particular Ustinovskiy proves that the $\text{HCF}_{\text{U}}$ on a compact Hermitian manifold preserves Griffiths non-negativity of the Chern curvature, generalizing the classical result that K\"ahler-Ricci flow preserves the positivity of the bisectional holomorphic curvature (see e.g. \cite{M}). In \cite{U3} the author could prove stronger results showing that the $\text{HCF}_{\text{U}}$ preserves several natural curvature positivity conditions besides Griffiths positivity. In \cite{U2} Ustinovskiy focuses on complex homogeneous manifolds and proves that the finite dimensional space of induced metrics (which are not necessarily invariant) is preserved by the $\text{HCF}_{\text{U}}$.

Given a connected complex Lie group $G$ acting transitively, effectively and holomorphically on a complex manifold $M$, a simple observation shows that $M$ does not carry any $G$-invariant Hermitian metric unless the isotropy is finite. More recently in \cite{LPV} Lafuente, Pujia and Vezzoni considered the behaviour of a general HCF in the space of left-invariant Hermitian metrics on a complex unimodular Lie group.

In this work we focus on C-spaces, namely compact simply connected complex manifolds $M$ which are homogeneous under the action of a compact semisimple Lie group $G$. By classical results $M$ fibers over a flag manifold $N$ with a complex torus as a typical fiber $F$; in particular we consider the case where $N$ is a product of compact Hermitian symmetric spaces and $F$ is non-trivial (so that $M$ does not carry any K\"ahler metric). While the analysis of a generic HCF on these manifolds seems to be out of reach, some special flows may deserve attention. In view of the classification results obtained in \cite{FGV} for C-spaces carrying invariant SKT metrics, the pluriclosed flow can be investigated only in very particular cases of such spaces (see also \cite{B} for the analysis of the pluriclosed flow on compact locally homogeneous surfaces and \cite{AL} for the case of left-invariant metrics on Lie groups). On the other hand the $\text{HCF}_{\text{U}}$ is geometrically meaningful and can be dealt with more easily. Actually we are able to write down the flow equations in the space of $G$-invariant metrics in a surprisingly simple way. In particular we prove the remarkable fact that the flow can be described by an induced flow on the base, which depends only on the initial conditions on the base itself, and an induced flow on the fiber. Moreover the maximal existence domain is bounded above and we can provide a precise description of the limit metric. Indeed we see that the kernel of the limit metric defines an integrable distribution whose leaves coincide with the orbits of the complexification of a suitable normal subgroup $S$ of $G$. When the leaves of this foliation are closed, we can prove the Gromov-Hausdorff convergence of the space to a lower dimensional Riemannian homogeneous space. We are also able to estabilish the existence and the uniqueness up to homotheties of invariant static metrics for the $\text{HCF}_{\text{U}}$. These metrics turn out to be particularly meaningful when $S=G$, as in this case the normalized flow with constant volume converges to one of them.

The work is organized as follows. In Section 2 we give some prelimiary notions on C-spaces and the $\text{HCF}_{\text{U}}$. In Section 3 we compute the invariant tensors involved in the flow equations and in Section 4 we prove our main result, which is summarized in Theorem \ref{Main} and Proposition \ref{GH}.

\subsection*{Notation} Throughout the following we will denote Lie groups by capital letters and the corresponding Lie algebras by gothic letters. The Cartan Killing form of a Lie algebra will be denoted by $\kappa$. \par \medskip
\subsection*{Acknowledgements} The authors heartily thank Yury Ustinovskiy and Marco Radeschi for many valuable discussions and remarks as well as Luigi Vezzoni and Daniele Angella for their interest.
\section{Preliminaries}

We consider a compact simply connected complex manifold $(M,J)$ with $\dim_{\small{\bC}} M = m$ and we assume that it is homogeneous under the action of a compact semisimple Lie group $G$ of biholomorphisms, namely $M = G/L$ for some compact subgroup $L\subset G$. By Tits fibration theorem the manifold $M$ fibers $G$-equivariantly onto a flag manifold $N:= G/K$, say $\pi:M\to G/K$, and the  manifold $N$ can be endowed with a $G$-invariant complex structure $I$ so that the fibration $\pi: (M,J)\to (N,I)$ is holomorphic. Since $M$ is supposed to be simply connected, the typical fiber $F:= K/L$ is a complex torus of complex dimension $k$ (see e.g. \cite{A}). Such a homogeneous complex manifold, which will be called a simply connected C-space (see \cite{W}), is not K\"ahler if the fiber $F$ is not trivial. \par 
In this work we will assume that the base of the Tits fibration is a product of irreducible Hermitian symmetric spaces of complex dimension at least two. More precisely, if we write $G$ as the (local) product of its simple factors, say $G=G_1\cdot\ldots\cdot G_{s}$, then $K$ also splits accordingly as $K = K_1\cdot\ldots\cdot K_{s}$ with $K_i\subset G_i$ and  $(G_i,K_i)$ is a Hermitian symmetric pair with $\dim_{\small{\bC}} G_i/K_i = n_i$. Note that  $m = k + \sum_{i=1}^{s} n_i$. \par 
At the level of Lie algebras we can write the Cartan decomposition $\gg_i = \gk_i\oplus \gn_i$  for each $i=1,\ldots,s$. We recall that the center of $\gk_i$ is one-dimensional and spanned by an element $Z_i$ so that $\gk_i = \gs_i \oplus \bR Z_i$, $\gs_i$ being the semisimple part of $\gk_i$, and  $\ad(Z_i) = I|_{\gn_i}$. We can now write the following decompositions 
$$\gl = \bigoplus_{i=1}^s \gs_i \oplus \gb, \quad \gg = \gl \oplus \gf \oplus \gn,\quad  $$
where $\gn:= \bigoplus_{i=1}^s\gn_i$ and $\gb,\gf$ are abelian subspaces of $\gz(\gk) = \bigoplus_{i=1}^s \bR Z_i$ with $\kappa(\gb,\gf) =0$. Note that $\gf$ and $\gm := \gf \oplus \gn$ identify with the tangent spaces to the fiber and to $M$ respectively. 

 Since the fibration $\pi:M\to N$ is holomorphic, the complex structure $J\in \End(\gm)$ can be written as $J = I_F + I$, where $I_F$ is a totally arbitrary complex structure on the fibre $F$. \par 

We now consider a $G$-invariant Hermitian metric $h$ on $M$, which can be seen as an $\ad(\gl)$-invariant Hermitian inner product on $\gm$. As the $\gs_i$'s are not trivial, we have that $\gl$ acts non-trivially on $\gn$ and trivially on $\gf$, therefore  $h(\gf,\gn)=0$. In particular the restriction $h|_{\gf\times \gf}$ is an arbitrary Hermitian metric. \par Moreover, the $\ad(\gl)$-modules $\gn_i$ are mutually non-equivalent, hence $h(\gn_i,\gn_j)=0$ if $i\neq j$. \par 
If $\gn_i$ is $\gs_i$-irreducible, then Schur Lemma implies that $$h|_{\gn_i\times \gn_i} := -h_i\kappa_i,$$ 
where $h_i\in \bR^+$ and $\kappa_i$ denotes the Cartan-Killing form on $\gg_i$. Note that this is always the case unless $\gg_i = \so(n+2)$ and $\gk_i = \so(2) \oplus \so(n)$ ($n\geq 3$). Throughout the following we will assume that none of the Hermitian factors of the basis $N$ is a complex quadric. \par
 
Given a $G$-invariant Hermitian metric $h$ on $M$, we can consider the associated Chern connection $\nabla$, which is the unique metric connection ($\nabla h = 0$) that leaves $J$ parallel ($\nabla J=0$) and whose torsion $T$ satisfies for $X,Y$ tangent vectors
$$T(JX,Y)= T(X,JY) = JT(X,Y).$$ The curvature tensor $R_{XY} = [\n_X,\n_Y]-\n_{[X,Y]}$ of $\nabla$ gives rise to different Ricci tensors and we are mainly interested in the second Chern-Ricci tensor $S$ which is given by 
$$S(X,Y) =  \sum_{a=1}^ {2m} h(JR_{e_a,Je_a}X,Y),$$
where $\{e_1,\ldots,e_{2m}\}$ denotes an $h$-orthonormal basis. In \cite{ST} Streets and Tian have introduced a family of Hermitian curvature flows  on any complex manifold given by 
\beq\label{flow} h_t' = -S(h) + Q(T),\eeq
where $Q(T)$ is symmetric, $J$-invariant tensor which is an arbitrary quadratic expression involving the torsion $T$. In \cite{ST} the authors proved the short time existence for all these flows for any initial Hermitian metric. In \cite{U} Ustinovskiy considered a special Hermitian flow where the quadratic term $Q(T)$ is given in complex coordinates by 
\beq\label{Q}Q(T)_{i\overline j} = -\frac 12 h^{m\overline n}h^{p\overline s}T_{mp\overline j}T_{\overline n \overline s i}.\eeq 
For the corresponding Hermitian curvature flow ($\text{HCF}_{\text{U}}$) Ustinovskiy could prove several important properties, in particular that it preserves the Griffiths positivity of the initial metric. \par
We now focus on this Hermitian flow on the class of homogeneous compact complex manifolds $(M,J)$ we have introduced in this section. In particular we note that the flow evolves along invariant Hermitian metrics whenever the initial metric is so.

\section{The computation of the tensors  }

In this section we compute the Ricci tensor $S$ and the quadratic expression $Q(T)$ in \eqref{Q} for a $G$-invariant Hermitian metric $h$ on the complex manifold $M = G/L$. Throughout the following we keep the notation introduced in the previous section. \par 
We choose a maximal abelian subalgebra $\ga_i\subset \gk_i$ of $\gg_i$ ($i=1,\ldots, s$). The complexification $\ga^{c}\subset \gg^{{c}}$ where $\ga := \bigoplus_{i=1}^{s}\ga_i$, gives a Cartan subalgebra of $\gg^c$ and we will denote by $R$ the associated root system of $\gg^c$. We denote by $R_i$ the subset of $R$ given by all the roots whose corresponding root vectors belong to $\gg_i^c$ for $i=1,\ldots, s$, so that $R= R_1\cup \ldots\cup R_{s}$. For each $i=1,\ldots, s$ we have 
$$\gk_i^c = \ga_i^c\oplus \bigoplus_{\a\in R_{\gk_i}}\gg_\a,\qquad \gn_i^c =  \bigoplus_{\a\in R_{\gn_i}} \gg_\a$$
so that $R_i = R_{\gk_i}\cup R_{\gn_i}$. Moreover the invariant complex structure $I|_{\gn_i}$ on $\gn_i$ determines an invariant ordering of $R_{\gn_i} = R_{\gn_i}^+\cup R_{\gn_i}^-$ with $I(v) = \pm \sqrt{-1} v$ for every $v\in \gg_\a$ with  $\a\in R_{\gn_i}^{\pm}$.

We will use a standard Chevalley basis $\{E_\a\}_{\a\in R}$ of $\gg^c$ with 
$$\gg^c = \ga^c \oplus \bigoplus_{\a\in R} \bC E_\a,\quad  \overline{E_\a} = - E_{-\a},\quad \kappa(E_\a,E_{-\a}) = 1, $$
$$[E_\a,E_{-\a}] = H_\a,$$
where $\kappa(H_\a,v) = \a(v)$ for every $v\in \ga^c$.\par 
Let $h$ be an invariant Hermitian metric on $M$, i.e. an $\ad(\gl)$-invariant symmetric Hermitian inner product on $\gm$. We recall that $h|_{\gn_i\times \gn_i} := -h_i\kappa_i$ and by the $\ad(\ga)$-invariance be have 
$$ h(E_\a,\overline{E_\b}) = \begin{cases} 0 & \text{if $\a\neq\b$} \\ 
h_i & \text{if $\a=\b\in R_i$.}\end{cases}$$ 
In the complexified tangent space $\gf^c$ of the fiber we fix a complex basis $\mathcal V := \{V_1,\ldots,V_k\}$ of $\gf^{10}$. We also put $h_{a\bar b} := h(V_a,\overline{V_b})$ and $H:= (h_{a\bar b})_{a,b=1,\ldots,k}$.

If we split $\gm^{c} = \gm^{10}\oplus \gm^{01}$ with respect to the extension of $J$ on $\gm^c$, the torsion $T$ can be seen as an element of $\Lambda^2(\gm^*)\otimes \gm$ and it satisfies
\beq \label{T}T(\gm^{10},\gm^{01}) = 0.\eeq
The Chern connection $\nabla$ is completely determined by the corresponding Nomizu's operator $\Lambda\in \gm^*\otimes \End(\gm)$ (see e.g.\cite{KN}), which is defined as follows: for $X,Y\in\gm$ we set $\Lambda(X)Y\in \gm$ to be so that at $o:= [L]\in M$
	$$(\Lambda(X)Y)^*|_o = (\nabla_{X^*}Y^*-[X^*,Y^*])|_o,$$
where for every $Z\in \gm$ we denote by $Z^*$ the vector field on $M$ induced by the one-parameter subgroup $\exp(tZ)$. Since $\nabla$ preserves $J$, the Nomizu's operator satisfies $\Lambda(v)(\gm^{10})\subseteq \gm^{10}$ for every $v\in \gm^c$. Now we recall that the torsion $T$ can be expressed  as follows: for $X,Y\in \gm$ (see e.g. \cite{KN})
	\beq\label{tor} T(X,Y) = \L(X)Y-\L(Y)X - [X,Y]_\gm.\eeq 
	Therefore using \eqref{T} and \eqref{tor} we see that 
\beq\label{Nomizu} \Lambda(A)\overline B = [A,\overline B]^{01}\quad \forall\ A,B\in \gm^{10}.\eeq

\begin{lemma}\label{Lambda} 
	Given $v\in\gf^c$, $w\in\gf^{10}$, $\alpha,\beta\in R_{\gn_i}^+$, $\gamma\in R_{\gn_j}^+$ with, $i,j=1,\dots,s$, $i\neq j$, we have
	\begin{itemize}
		\item[a)] $\Lambda(E_\alpha)E_\beta = 0$,\quad $\Lambda(E_\alpha)E_{\pm\gamma} = 0$ and $\Lambda(E_\alpha)\overline w=0 $;
		\item[b)] $\Lambda(E_\alpha)E_{-\beta} = 0$ for $\beta\neq \alpha$;\quad $\Lambda(E_\alpha)E_{-\alpha} =H_\a^{01}$; 
		\item[c)] $\Lambda(E_\alpha)w= \frac{1}{h_i}h(w,H_\alpha) E_\alpha$;
		\item[d)] $\Lambda(v) = \ad(v)$.
	\end{itemize}
\end{lemma}
\begin{proof}
	a)  If $A\in\gn^{10}$ and $w\in\gf^{10}$, we have by (\ref{Nomizu}) and using the fact that $[\gn_i,\gn_i]\subseteq\gk_i$:
	\begin{equation*}
	\begin{split}
	&h(\L(E_\a)E_\b,\overline{A})=-h(E_\b,\L(E_\a)\overline{A})=-h(E_\b,[E_\a,\overline{A}]^{01})=0,\\
	&h(\L(E_\a)E_\b,\overline{w})=-h(E_\b,\L(E_\a)\overline{w})=-h(E_\b,[E_\a,\overline{w}]^{01})=\a(\overline{w})h(E_\b,E_\a^{01})=0,
	\end{split}
	\end{equation*}
	therefore $\L(E_\a)E_\b=0$. Similarly we get $\L(E_\a)E_{\pm\g}=0$. Finally
	$$
	\L(E_\a)\overline{w}=[E_\a,\overline{w}]^{01}=-\a(\overline{w})E_\a^{01}=0.
	$$
	
	b) We have $\L(E_\a)E_{-\b}=[E_\a,E_{-\b}]^{01}$. If $\a\neq \b$, then $[E_\a,E_{-\b}]$ lies in the semisimple part of $\gk_i^c$, hence $[E_\a,E_{-\b}]^{01}= 0$. If $\a=\b$, $[E_\a,E_{-\a}]^{01}=H_\a^{01}$.
	
	c) Since we have, by a) and b), $h(\L(E_\a)w,E_{-\a})=-h(w,\L(E_\a)E_{-\a})=-h(w,H_\a)$, and $h(\L(E_\a)w,E_{-\b})=h(\L(E_\a)w,E_{-\g})=h(\L(E_\a)w,\overline{w'})=0$ (here $\b\in R_{\gn_i}^+$, $\b\neq\a$, $w'\in\gf^{10}$), the assertion follows.
	
	d) First we have, if $w,w'\in\gf^{10}$, $\L(w)\overline{w'}=[w,\overline{w'}]^{01}=0=\ad(w)\overline{w'}$ and $\L(w)E_{-\a}=[w,E_{-\a}]^{01}=-\a(w)E_{-\a}=\ad(w)E_{-\a}$. Furthermore,
	\begin{equation*}
	\begin{split}
	& h(\L(w)E_{\a},\overline{w'})=-h(E_\a,\L(w)\overline{w'})=0=\a(w)h(E_\a,\overline{w'})=h(\ad(w)E_\a,\overline{w'}),\\
	& h(\L(w)E_\a,E_{-\b})=-h(E_\a,\L(w)E_{-\b})=\b(w)h(E_\a,E_{-\b})=\d_{\a,\b}\a(w)h(E_\a,E_{-\a})\\
	&\qquad\qquad\qquad\qquad=\a(w)h(E_\a,E_{-\b})=h(\ad(w)E_\a,E_{-\b}),
	\end{split}
	\end{equation*}
	therefore $\L(w)E_\a=\ad(w)E_\a$. At this point it is easily seen that $\L(w)w'=0=\ad(w)w'$, so $\L(w)=\ad(w)$. Conjugation yields also $\L(\overline{w})=\ad(\overline{w})$, hence assertion d) follows.
\end{proof}

Using the previous Lemma we can compute the torsion tensor as follows.
\begin{lemma}\label{Tor2} 
	Given $\alpha,\beta\in R_{\gn_i}^+$, $\gamma\in R_{\gn_j}^+$ with $i\neq j$, and $w,w'\in\gf^{10}$ we have
	\begin{itemize}
		\item[a)] $T(E_\a,E_\g) = T(E_\a,E_{\b}) = 0$; 
		\item[b)] $T(w,E_\a) = -\frac{1}{h_i}h(w,H_\a) E_\a$;
		\item[c)] $T(w,w') = 0$.
	\end{itemize}
\end{lemma}
\begin{proof}
	a) From Lemma \ref{Lambda}(a) and (\ref{tor}), we see that
	$$
	T(E_\a,E_\g)=-[E_\a,E_\g]_{\gm^c},\qquad T(E_\a,E_\b)=-[E_\a,E_\b]_{\gm^c}.
	$$
	Now $\a+\g$ cannot be a root, therefore $[E_\a,E_\g]=0$. On the other hand, since $G_i/K_i$ is symmetric, $[E_\a,E_\b]\in\gk_i\cap \gn_i=0$. This proves the assertion.
	
	b) By Lemma \ref{Lambda}(c)-(d) and (\ref{tor}) we have
	\begin{equation*}
	\begin{split}
	T(w,E_\a)&=\L(w)E_\a-\L(E_\a)w-[w,E_\a]_{\gm^c}\\
	&=\a(w)E_\a-\frac{1}{h_i}h(w,H_\a)\,E_\a-\a(w)E_\a=-\frac{1}{h_i}h(w,H_\a)\,E_\a.
	\end{split}
	\end{equation*}
	
	c) It follows immediately from Lemma \ref{Lambda}(d).
\end{proof}

\noindent In order to compute the curvature tensor $R$, we use the general formula given in \cite[p. 192]{KN} (cf. also \cite{P}): for $X,Y\in \gm$ 
\beq\label{Chern_Curvature}
R(X,Y)=[\L(X),\L(Y)]-\L([X,Y]_{\gm})-\ad([X,Y]_\gl).
\eeq

\begin{lemma}\label{Curvature}
	Given $\alpha,\beta\in R_{\gn_i}^+$, $\g\in R_{\gn_j}^+$, $i\neq j$, $v_1,v_2\in\gf^c$, $w\in\gf^{10}$, we have
	\begin{itemize}
		\item[a)] $R(E_\a,\overline{E_\a})E_\b=\L(E_\a)\L(\overline{E_\a})E_\b+\b(H_\a)E_\b$,
		\item[b)] $R(E_\a,\overline{E_\a})E_\g=0$,
		\item[c)] $R(E_\a,\overline{E_\a})w=\frac{1}{h_i}h(w,H_\a)\,\overline{H_\a^{01}}$,
		\item[d)] $R(v_1,v_2)=0$.
	\end{itemize}
\end{lemma}
\begin{proof}
	a) Using (\ref{Chern_Curvature}) and Lemma \ref{Lambda} we have
	\begin{equation*}
	\begin{split}
	R(E_\a,\overline{E_\a})E_\b&=[\L(E_\a),\L(\overline{E_\a})]E_\b-\L([E_\a,\overline{E_\a}]_{\gm^c})E_\b-\ad([E_\a,\overline{E_\a}]_{\gl^c})E_\b\\
	&=\L(E_\a)\L(\overline{E_\a})E_\b+\L((H_\a)_{\gm^c})E_\b+\ad((H_\a)_{\gl^c})E_\b\\
	&=\L(E_\a)\L(\overline{E_\a})E_\b+[H_\a,E_\b]\\
	&=\L(E_\a)\L(\overline{E_\a})E_\b+\b(H_\a)E_\b.
	\end{split}
	\end{equation*}
	\indent b) is proved similarly.
	
	c-d) We have, using (\ref{Chern_Curvature}), Lemma \ref{Lambda} and the fact that $\gf^c$ is abelian:
	\begin{equation*}
	\begin{split}
	R(E_\a,\overline{E_\a})w&=-\L(\overline{E_\a})\L(E_\a)w+\L((H_\a)_{\gm^c})w+\ad((H_\a)_{\gl^c})w\\
	&=-\frac{1}{h_i}h(w,H_\a)\,\L(\overline{E_\a})E_\a+[H_\a,w]
	=\frac{1}{h_i}h(w,H_\a)\,\overline{H_\a^{01}}
	\end{split}
	\end{equation*}
	and, for similar reasons,
	$$
	R(v_1,v_2)=[\L(v_1),\L(v_2)]=[\ad(v_1),\ad(v_2)]=\ad([v_1,v_2])=0.
	$$
\end{proof}

We can now compute the second Chern-Ricci tensor $S$. If $\b\in R_{\gn_i}^+$ we have by Lemmas \ref{Lambda}, \ref{Curvature}:
\begin{equation*}
\begin{split}
S(E_\b,&\overline{E_\b})=\sum_{j=1}^{s}\sum_{\a\in R_{\gn_j}^+}\frac{1}{h_j}h(R(E_\a,\overline{E_\a})E_\b,\overline{E_\b})
\\
&=\sum_{\a\in R_{\gn_i}^+}\frac{1}{h_i}h(R(E_\a,\overline{E_\a})E_\b,\overline{E_\b})
=-\sum_{\a\in R_{\gn_i}^+}\frac{1}{h_i}h(\L(\overline{E_\a})E_\b,\L(E_\a)\overline{E_\b})+\sum_{\a\in R_{\gn_i}^+}\b(H_\a)\\
&=-\frac{1}{h_i}h(\overline{\L(E_\b)E_{-\b}},\L(E_\b)E_{-\b})+\sum_{\a\in R_{\gn_i}^+}\b(H_\a)
=-\frac{1}{h_i}h(\overline{H_\b^{01}},H_\b^{01})+\sum_{\a\in R_{\gn_i}^+} \b(H_\a)\\
&=-\frac{1}{h_i}h(\overline{H_\b^{01}},H_\b^{01})+\frac{1}{2}
\end{split}
\end{equation*}
where we have used that $\sum_{\a\in R_{\gn_i}^+}H_\a=-\frac{\sqrt{-1}}{2}Z_i$ and that $\b(Z_i)=\sqrt{-1}$. Similarly, for $a,b=1,\dots,k$,
\begin{equation*}
\begin{split}
S(V_a,\overline{V_b})&=\sum_{j=1}^{s}\sum_{\a\in R_{\gn_j}^+}\frac{1}{h_j}h(R(E_\a,\overline{E_\a})V_a,\overline{V_b})
\\
&=\sum_{j=1}^{s}\sum_{\a\in R_{\gn_j}^+}\frac{1}{h_j^2}h(V_a,H_\a)h(\overline{V_b},\overline{H^{01}_\a})
=\sum_{j=1}^{s}\sum_{\a\in R_{\gn_j}^+}\frac{1}{h_j^2}h(V_a,H_\a)\overline{h(V_b,H_\a)}.           
\end{split}
\end{equation*}

We now compute the tensor $Q(T)$. For $\a\in R^+_{\gn_i}$, $i=1,\dots,s$, set
$$
e_\a:=\frac{E_\a}{\sqrt{h_i}}.$$
We fix a $h$-orthonormal basis $\{e_a\}_{a=1,\ldots,k}$ of $\gf^{10}$. 
In the following, we shall use the greek letters $\a,\b,\dots$ as indices varying among the positive roots, while the lowercase latin letters $a,b,\dots$ will denote indices varying in the set $\{1,\dots,k\}$. Finally, latin capital letters $A,B,\dots$ will vary both among the positive roots and the elements of the set $\{1,\dots,k\}$. So we have, if $\b\in R^+_{\gn_i}$,
\begin{equation*}
\begin{split}
Q(T)(e_\b,\overline{e_\b})&=-\frac{1}{2}\sum_{A,B} T^\b_{AB}\overline{T}^\b_{AB}
=-\frac{1}{2}\sum_{\a,b} T^\b_{\a b}\overline{T}^\b_{\a b}-\frac{1}{2}\sum_{\a,b}T^\b_{b\a}\overline{T}^\b_{b\a}\\
&=-\sum_{\a,b}T^\b_{\a b}\overline{T}^\b_{a b}=-\sum_b\frac 1{h_i^2}|h(e_b,H_\b)|^2,
\end{split}
\end{equation*}
whence 
$$Q(T)(E_\b,\overline{E_\b})=-\frac{1}{h_i}\sum_b|h(e_b,H_\b)|^2 = 
-\frac{1}{h_i}h(\overline{H_\b^{01}},H_\b^{01}).$$
The last equality in the previous formula holds noting that, if one writes $H_\b^{01}=\sum_a\l_a\overline{e}_a$ for suitable $\l_a\in\C$, then 
$$\sum_{b=1}^k |h(e_b,H_\b)|^2= \sum_{b=1}^k|\l_b|^2=  h(\overline{H_\b^{01}},H_\b^{01}). $$

Moreover, using Lemma \ref{Tor2} we immediately see that $Q(T)(\gf,\gf) = 0$ and $Q(T)(\gf,\gn)=0$.\par 
 We can now write an expression for the tensor 
$$\mathcal{K}(h) := -S(h) + Q(T),$$
which governs the flow \eqref{flow}. Namely,
\beq \label{K} \begin{cases} 
	\mathcal{K}(h)(E_\a,\overline{E_\a}) = -\frac 12 ,\quad \a\in R_{\gn}^+;\\ \mathcal{K}(h)(\gf,\gn) =0;  \\
\mathcal{K}(h)(V_a,\overline{V_b}) = -\sum_{j=1}^{s}\frac{1}{h_j^2}\sum_{\a\in R_{\gn_j}^+}h(V_a,H_\a)\overline{h(V_b,H_\a)}. \end{cases}\eeq

\section{The analysis of the flow and static metrics}

Starting from an invariant Hermitian metric $h_o$, the unique solution to the flow equation \eqref{flow} consists of $G$-invariant Hermitian metrics on $M$. Therefore using \eqref{K} we can write the flow equations as follows 

\beq\label{InvHCF_gen}
\begin{cases}
	h_{a\overline{b}}'=-\sum_{j=1}^{s}\frac{1}{h_j^2}\sum_{\a\in R_{\gn_j}^+}h(V_a,H_\a)\overline{h(V_b,H_\a)},\quad\text{for $a,b=1,\dots,k$},\\
	h_i'=-\frac{1}{2},\quad\text{for $i=1,\dots,s$}.
\end{cases}
\eeq

In order to write the equations in (\ref{InvHCF_gen}) relative to the fibre in a nicer form, set $R^+_{\gn_j}:=\{\a^j_1,\dots,\a^j_{n_j}\}$ for $j=1,\dots,s$. 
We note that for every $\a\in R_{\gn_j}^+$ ($j=1,\ldots,s$) we have $H_\a = -\frac{\sqrt{-1}}{2n_j}Z_j\ ({\rm{mod}}\ \gs_j)$, hence  we can find coefficients $c^j_i$, $j=1,\ldots,s$, $i=1,\ldots,k$ so that 
\beq\label{cijl}
H_{\a^j_i}^{01}=\sum_{l=1}^k c^j_{l}\overline{V_{l}}\qquad i=1,\ldots,n_j,\ j=1,\ldots,s.
\eeq
Thus 
\begin{equation*}
h_{a\bar b}'= -\sum_{j=1}^s \frac 1{h_j^2}\sum_{l,m=1}^k \left( n_j
c^j_{l}\overline{c^j_{m}}\right) h_{a\bar l} h_{m\bar b} =  
-\sum_{j=1}^s \frac 1{h_j^2}\left( H\Gamma^j H\right)_{a\bar b},
\end{equation*}
where we have set for $j=1,\ldots,s$
$$(\Gamma^j)_{l\bar m} := n_j c^j_{l}\overline{c^j_{m}}.$$

Therefore we can write (\ref{InvHCF_gen}) as
\beq\label{InvHCF_gen_matrix}
\begin{cases}
	H'=-H\Gamma H\\
	h_i'=-\frac{1}{2},\quad\text{for $i=1,\dots,s$},
\end{cases}
\eeq
where 
$$\Gamma(t) := \sum_{j=1}^s \frac 1{h_j^2(t)} \Gamma^j.$$
We note that $\Gamma$ is positive semidefinite as each $\Gamma^j$ is so, $j=1,\ldots,s$.

The metric $h_o$ can be fully described by $s$ positive numbers $A_1,\ldots,A_{s}$ where 
\beq\label{Ai}h_o|_{\gn_i\times \gn_i} = -A_i\kappa\eeq 
and by a positive definite Hermitian $k\times k$ matrix $H_o$, which represents $h_o|_{\gf\times\gf}$ w.r.t. the basis $\mathcal V$. From \eqref{InvHCF_gen_matrix} we immediately see that 
\beq\label{solhi} h_i(t) = -\frac 12 t + A_i,\quad i=1,\ldots, s.\eeq
If we set $A:= \min_{i=1,\ldots,s}\{A_i\}$, we see that $h_i(t)$ are all positive when $t\in [0,2A)$. The flow equation boils down to 
\beq \begin{cases} H' = -H\Gamma H\\ H(0) = H_o\end{cases}\eeq 
that can be explicitely integrated to 
\beq\label{H} H^{-1}(t) = H_o^{-1} + \int_0^t \Gamma(u)\, du.\eeq
Note that the righthandside of \eqref{H} is positive definite for all $t\in [0,2A)$, therefore the solution $h(t)$ to the $\text{HCF}_{\text{U}}$ exists on $[0,2A)$. Moreover we notice that the maximal existence domain of $h(t)$ is of the form $(-r,2A)$, where $r\in (0,+\infty]$. \par 
In order to analyze the behaviour of the metric along the fiber when $t$ approaches the limit $2A$, we simply observe that \eqref{InvHCF_gen} implies
\beq\label{decreasing}
h(v,\bar v)'=-\sum_{j=1}^s\frac{1}{h_j^2}\sum_{\a\in R^+_{\gn_j}}|h(v,H_\a)|^2\leq 0,\qquad\text{for any $v\in\gf^{10}$},
\eeq
therefore
$
\lim_{t\rightarrow 2A}h(v,\bar v)
$
exists and is non-negative. Thus when $t\rightarrow 2A$ the metric along the fiber converges to a positive semidefinite Hermitian form $\hat h$. 
\begin{proposition}\label{prop} There is a compact normal subgroup $S$ of $G$ such that the orbits of $S^c$ are the leaves of the distribution defined by the kernels of $\hat{h}$.
\end{proposition} 	
	
\begin{proof} We consider the distribution $\mathcal Q$ on $M$ which is defined for $x\in M$  by $\mathcal Q_x :=\{v\in T_xM|\ \hat h_x(v,w)=0,\ \forall  w\in T_xM\}$. It is clear that $\mathcal Q$ is $G$-invariant and $J$-stable, so that it is enough to study it at $o:= [L]\in M=G/L$, where we can see $\gq:= \mathcal Q_o$ as a $J$-stable subspace of $\gm$. We write $\gq^c = \gq^{10}\oplus \gq^{01}$.  We rearrange the indices so that $A=A_1=\ldots=A_p<A_i$ for $i=p+1,\dots,s$ and we define $\mathcal{Z}_p$ to be the complex subspace of $\gf^{01}$ generated by $Z_1^{01},\dots,Z_p^{01}$. The limit form $\hat h$ is described by a pair $(\hat h^N,\hat h^F)$, where $\hat h^N$ is a $\Ad(L)$-invariant Hermitian form on $\gn$ whose kernel is given by $\gn_1\oplus\ldots\oplus \gn_p$ and $\hat h^F$ is a positive semidefinite Hermitian form on $\gf$.
\begin{lemma}\label{q} $\gq^c = \gn_1^c\oplus\ldots\oplus \gn_p^c\oplus \mathcal Z_p\oplus \overline{\mathcal Z_p}$.
	\end{lemma}	
\begin{proof} It is enough to prove that $\mathcal Z_p = (\ker \hat h^F)^{01}$. Throughout the following we will identify $\gf^{01}$ with $\C^k$ by means of the basis $\overline{\mathcal{V}}$. Observe that \eqref{H} reads
	$$
	H^{-1}(t) = H_o^{-1} -\frac{2t}{A(t-2A)}\Theta_p +\sum_{j=p+1}^s\left(\int_0^t\frac{1}{h_j^2(u)}du\right)\Gamma^j,
	$$
	where $\Theta_p:=\sum_{j=1}^p\Gamma^j$. Since the image of $\Gamma^j$ is spanned by $Z_j^{01}$ for $j=1,\dots,k$, we see that $\Theta_p(\mathcal{Z}_p)\subseteq\mathcal{Z}_p$. Moreover, as each $\Gamma^j$ is Hermitian positive semidefinite, we have that $\ker \Theta_p = \bigcap_{j=1}^p \ker \Gamma^j = \mathcal Z_p^\perp$, where the orthogonal space $\mathcal Z_p^\perp$ is taken with respect to the standard Hermitian structure on $\mathbb C^k$. This implies that $\ker \Theta_p \cap \mathcal Z_p = \{0\}$ and therefore $\Theta_p(\mathcal Z_p) = \mathcal Z_p$. We also observe that $\Theta_p$ is diagonalizable and therefore its image $\mathcal Z_p$ is the sum of all eigenspaces with non-zero eigenvalues. If we set $q:= \dim_{\small{\bC}}\mathcal Z_p$, there exists $U\in \U(k)$ so that 
$$\Delta:= U\Theta_p \bar U^T$$
is the diagonal matrix ${\rm{diag}}(\mu_1,\ldots,\mu_q,0,\ldots,0)$, $\mu_i\neq 0$ for $i=1,\ldots,q$. Then if we set $H_u(t) := UH(t)\bar U^T$ we have for $t\in [0,2A)$
\beq\label{flow_u} H_u^{-1}(t) = \Lambda(t) - \frac{2t}{A(t-2A)}\Delta,\eeq
where $\Lambda(t)$ is positive definite for all $t\in[0,2A)$ and, when $t\to 2A$, it converges to a positive definite matrix  
$$\hat \Lambda := \lim_{t\rightarrow 2A} \Lambda(t).$$
We have 
$$H_u(t) = \frac{ {\rm {adj}}(H_u^{-1}(t))}{\det H_u^{-1}(t)   }$$
and using \eqref{flow_u} we see that
$$
\lim_{t\rightarrow 2A}(2A-t)^q \det H_u^{-1}(t)
=4^q\mu_1\cdot\ldots\cdot\mu_q \det \hat{\L}_o> 0
$$
where $\hat{\L}_o$ is the minor of $\hat{\L}$ obtained intersecting the last $k-q$ rows and the last $k-q$ columns. Moreover
$$
\lim_{t\rightarrow 2A}(2A-t)^q {\rm {adj}}(H_u^{-1}(t))_{a\bar b}=0,\quad\forall\,a\in\{1,\dots,q\},\;\forall\,b\in\{1,\dots,k\},
$$
hence
$$
\lim_{t\rightarrow 2A} (H_u(t))_{a\bar b}=0,\quad\forall\,a\in\{1,\dots,q\},\;\forall\,b\in\{1,\dots,k\},
$$
while for $a,b=q+1,\dots,k$, 
$$
\lim_{t\rightarrow 2A} (H_u(t))_{a\bar b}
=(\hat{\L}_o^{-1})_{(a-q)\overline{(b-q)}}.
$$
Thus we have that
$$
\lim_{t\rightarrow 2A}H_u(t)
 = \left(\!
\begin{array}{cc}
0 & 0\\
0 & \hat{\L}_o^{-1}
\end{array}\!
\right)\!,
$$
where $\hat{\L}_o^{-1}$ is $(k-q)\times (k-q)$ and positive definite. The claim follows.
\end{proof} 
We define $U:= G_1\cdot\ldots\cdot G_p$.  We observe that the universal complexification $U^c$ acts on $M$ and that the $U^c$-orbits define a $G$-invariant foliation. At the point $o\in M$ we have that $T_o(U^c\cdot o) = \gn_1\oplus\ldots\oplus\gn_p\oplus \Span\{ (Z_1)_\gf,\ldots,(Z_p)_\gf, I_F((Z_1)_\gf),\ldots,I_F((Z_p)_\gf)\} = \gq$ by Lemma \ref{q}. Therefore the $U^c$-orbits coincide with the leaves of $\mathcal Q$.	
\end{proof} 

\begin{remark}
Since the complex structure $I_F$ along the fiber is totally arbitrary, the $U^c$-orbits are not necessarily closed. Note that $G$ acts transitively on the set of $U^c$-orbits and therefore one such orbit is closed if and only if all are so.
\end{remark}

\subsection{Static metrics} We say that an invariant Hermitian metric $h$ on $M$ is {\em static} for the $\text{HCF}_{\text{U}}$ if there exists $\l\in\R$ such that
\beq\label{static}
-\mathcal{K}(h) = \l h.
\eeq
In terms of algebraic data on $\gg$ this equation becomes
\beq\label{static_matrix}
\begin{cases}
\l H=H\Gamma H,\\
\l h_i=\frac 12, & \text{for $i=1,\dots,s$}.
\end{cases}
\eeq
This immediately implies  that \eqref{static} has no solution if $\l\leq 0$. On the other hand, if $\l>0$, the second equation in \eqref{static_matrix} gives
\beq\label{static1}
h_i=\frac{1}{2\l}>0,\quad\text{for $i=1,\dots,s$}.
\eeq
 Then the first equation in \eqref{static_matrix} reads    
\beq\label{static2}
H=\frac{1}{4\l}\Theta_s^{-1},
\eeq 
where we note  that $\Theta_s$ is invertible as it is shown in the proof of Lemma \ref{q}. 
Using \eqref{static1}, \eqref{static2} we can therefore define, for any $\l>0$, a static metric satisfying \eqref{static} for the chosen $\l$.

Suppose now that in \eqref{Ai} the initial conditions satisfy $A=A_1=\cdots=A_s$. Clearly in this case we have 
$
\lim_{t\rightarrow 2A}h_i(t)=0$ for $i=1,\dots,s
$
and 
$$\int_0^t \Gamma(u)du = \frac{2t}{A(2A-t)}\Theta_s,\qquad t\in [0,2A).$$
On the other hand we have 
$$
\lim_{t\rightarrow 2A} H(t)
  =\lim_{t\rightarrow 2A}\frac{{\rm{adj}}(H^{-1}(t))}{\det H^{-1}(t)}=0,
$$
as one can easily verify that 
\beq\label{lim}\begin{cases}
\lim_{t\rightarrow 2A} (2A-t)^k \det H^{-1}(t)
  =4^k\det\Theta_s>0,\\
\lim_{t\rightarrow 2A} (2A-t)^{k-1}{{\rm{adj}}(H^{-1}(t))}=4^{k-1}{\rm{adj}}(\Theta_s).
\end{cases}\eeq
Thus
$$
\lim_{t\rightarrow 2A}h(t)=0.
$$

We now consider the Hermitian metric $\tilde h(t)$ which is homothetic to $h(t)$ and has unitary volume, namely $\tilde h(t):= c(t) h(t)$ with $c(t) = (\vol_{h(t)}(M))^{-1/m}$. Now we see that there exists a positive constant $V$ so that 
$$
\vol_{h(t)}(M)=V\cdot \det H(t)\cdot \prod_{i=1}^s(2A-t)^{n_i}
=V\cdot\frac{(2A-t)^{m}}{(2A-t)^k\det H^{-1}(t)},$$
where we have used that $m=k+\sum_{i=1}^sn_i$. We can then write 
$$c(t)=\frac{\xi(t)}{2A-t},\qquad 
\xi(t):=\left(V^{-1}(2A-t)^k\det H^{-1}(t)\right)^{1/m}
$$ 
and note that 
by \eqref{lim}, 
$$
\lim_{t\rightarrow 2A}\xi(t)=\left(V^{-1}4^k\det\Theta_s\right)^{1/m}:= \xi>0.
$$  
We now have, for $i=1,\dots,s$,
$$
\lim_{t\rightarrow 2A}c(t)h_i(t)=\lim_{t\rightarrow 2A}\frac{\xi(t)}{2A-t}\cdot\frac{2A-t}{2}=\frac{\xi}{2}.
$$
while, using \eqref{lim}
$$ \lim_{t\rightarrow 2A}c(t)H(t)
=\frac{\xi}{4}\cdot \frac{{\rm {adj}}(\Theta_s)}{\det \Theta_s} = \frac \xi 4\cdot \Theta_s^{-1},
$$
Therefore the solution $\tilde h(t)$ to the normalized flow converges to the static metric satisfying  \eqref{static} with $\l=1/\xi$.\par \medskip

We can then formulate our first result as follows 
\begin{theorem}\label{Main} Let $M=G/L$ be a simply connected C-space with Tits fibration over a Hermitian symmetric space $N$ whose irreducible factors have complex dimension at least two and are not complex quadrics. Any $G$-invariant Hermitian metric $h$ on $M$ is determined by a pair $(h^N,h^F)$, where $h^N$ is an $\Ad(L)$-invariant Hermitian metric on $\gn$ and $h^F$ is an arbitrary Hermitian metric on the fiber $\gf$.\par 
 Given any initial invariant Hermitian metric $(h^N_o,h^F_o)$ on $M$, we have:
	\begin{itemize}
		\item[i)] the solution $h(t)$ to the $\text{HCF}_{\text{U}}$ is given by the pair $(h^N(t),h^F(t))$, where $h^N(t)$ 
	 depends only on $h^N_o$;
	 \item[ii)]  the maximal existence domain of $h^N(t)$ is an interval $(-\infty,T)$ with $0<T< +\infty$ and the solution $h^F(t)$ exists and is positive definite on an interval $(-r,T)$ with $0<r\leq+\infty$. When $t\to T$, the base $N$ collapses to a product of some Hermitian symmetric spaces and the metric $h(t)$ converges to a positive semidefinite Hermitian bilinear form $\hat h$. The distribution given by the kernel of $\hat h$ is integrable and its leaves coincide with the $U^c$-orbits for a suitable compact connected normal subgroup $U\subseteq G$;
	 \item[iii)] the manifold $M$ admits a unique (up to homotheties) invariant Hermitian metric which is static for the $\text{HCF}_{\text{U}}$. 
	 If $h_N(t)\rightarrow 0$ when $t\rightarrow T$, then $h(t)\rightarrow 0$ and the normalized flow with constant volume converges to a static metric.
	\end{itemize} 
\end{theorem}

The following proposition can be thought of as a complement of the main Theorem \ref{Main} when the $U^c$-orbits are closed. We keep the same notation as in Theorem \ref{Main} and we consider the foliation, again denoted by $\mathcal Q$, given by the $U^c$-orbits. If $d_t$ denote the distance on $M$ induced by the metric tensor $h_t$ ($t\in [0,T)$), we describe the Gromov-Hausdorff limit of the metric spaces $(M,d_t)$ when $t\to T$.
\begin{proposition}\label{GH} If the $U^c$-orbits are closed, the leaf space $\overline  M := M/\mathcal Q$ has the structure of a smooth homogeneous manifold $G/\hat U$ for some closed subgroup $\hat U\subseteq G$. Moreover the positive semidefinite tensor $\hat h$ induces a $G$-invariant Riemannian metric $\overline h$ on $\overline M$ with induced distance $\bar d$ and the metric spaces $(M,d_t)$ Gromov-Hausdorff converge to  $(\overline M,\bar d)$ when $t\to T$.  \end{proposition}
\begin{proof} Note that $G$ acts transitively on the leaves of the $G$-invariant foliation $\mathcal Q$. The closedness of the leaves implies that $\overline M$ is Hausdorff and it can be expressed as a coset space $G/\hat U$ for some closed subgroup $\hat U$ that contains both $U$ and $L$. At the level of Lie algebras we can write the decomposition
$$\gg = \hat\gu \oplus \hat\gm,$$
where the subspace $\hat\gm$ is defined as the $\kappa$-orthocomplement of $\hat\gu$. As $\gl\subset \hat\gu$ we have that $\hat\gm\subset \gm$. Moreover $\gu = \bigoplus_{i=1}^p\gg_i\subset \hat\gu$ implies that $\hat\gm\subseteq \bigoplus_{j=p+1}^s\gg_j$. We now note that the $G$-equivariant projection $\pi:M\to N$ maps the orbit $U^co$ onto the $U$-orbit $\prod_{i=1}^pN_i\times \prod_{j=p+1}^s\{[K_j]\}$, where $N_i:= G_i/K_i$ for $i=1\ldots s$, whence $\hat U\subseteq \prod_{i=1}^p G_i \times \prod_{j=p+1}^sK_j$.Then we can write 
$$\hat \gm = \bigoplus_{j=p+1}^s\gn_j \oplus \hat \gf,\quad \hat\gf \subseteq \bigoplus_{j=p+1}^s\mathbb R Z_j.$$
This implies in particular that $\Ad(\hat U)|_{\hat\gf} = \mbox{Id}$. Therefore the restriction $\hat h|_{\hat\gm\times\hat\gm}$ gives a positive definite $\Ad(\hat U)$-invariant inner product which descends to a $G$-invariant Riemannian metric $\bar h$ on $\overline M$. \par We now prove the last statement, namely that for $x,y\in M$ we have 
$\lim_{t\to T}d_t(x,y) = \bar d(p(x),p(y))$, where $\bar d$ is the distance induced by $\bar h$ and $p:M\rightarrow\overline{M}$ is the projection. 
We denote by $\gamma^t$ a minimizing geodesic for $h_t$ joining $x$ with $y$. As \eqref{solhi} and \eqref{decreasing} imply that $h_t(v,v)$ is a non-incerasing funtion of $t$ for every tangent vector $v$, we see that  
\begin{equation*}\begin{split}
d_t(x,y) &= \int_0^1 h_t(\frac{d\gamma^t}{ds},\frac{d\gamma^t}{ds})^{1/2}\ ds\geq 
\int_0^1\hat h(\frac{d\gamma^t}{ds},\frac{d\gamma^t}{ds})^{1/2}\ ds\\
&=  \int_0^1\bar h(\frac{dp\circ\gamma^t}{ds},\frac{dp\circ\gamma^t}{ds})^{1/2}\ ds  \geq \bar d(p(x),p(y)).
\end{split}\end{equation*}
This means that 
$$\liminf_{t\to T}d_t(x,y) \geq \bar d(p(x),p(y)). $$
On the other hand let $\gamma$ be a minimizing geodesic for $\bar h$ connecting $p(x)$ with $p(y)$. Let $\tilde \gamma$ be a lift of $\gamma$ starting at $x$ with ending point $\tilde y$. We choose a path $\eta$ in $p^{-1}(p(y))$ connecting $y$ with $\tilde y$. Then 
$$d_t(x,y)\leq d_t(x,\tilde y)+d_t(\tilde y,y)\leq 
\int_0^1h_t(\tilde \gamma'(s),\tilde \gamma'(s))^{1/2} ds + 
\int_0^1 h_t(\eta'(s),\eta'(s))^{1/2}ds.$$
Now 
$$\lim_{t\to T} \int_0^1 h_t(\eta'(s),\eta'(s))^{1/2} ds = 0 ,$$
while 
\begin{equation*}\begin{split}
\lim_{t\to T} \int_0^1h_t(\tilde \gamma'(s),\tilde \gamma'(s))^{1/2}ds &= 
\int_0^1\hat h(\tilde \gamma'(s),\tilde \gamma'(s))^{1/2}ds\\
&=\int_0^1\bar h(\gamma'(s),\gamma'(s))^{1/2}ds = \bar d(p(x),p(y)).
\end{split}\end{equation*}
This implies that 
$$\limsup_{t\to T} d_t(x,y) \leq \bar d(p(x),p(y))$$
and this concludes the proof.   \end{proof} 
\begin{remark} Note that the homogeneous space $\overline{M}$ might not carry any ($G$-invariant) complex structure.
\end{remark}

\subsection{Example} A C-space $M$ will be called  of {\it Calabi-Eckmann type} if $M = G/L$, where $G = G_1\cdot G_2$, $L = L_1\cdot L_2$ and $L_i$ is the semisimple part of $K_i$ for $i=1,2$ (see \cite{CE},\cite{W}). The space $M$, which is $T^2$-bundle over  a product of two Hermitian symmetric spaces, can be endowed with many invariant complex structures as described in section 2. We consider now a more general C-space $M$ which is given by a product of C-spaces of Calabi-Eckmann type  $M_1,M_3,\ldots,M_{2k-1}$ with $M_i = G_i\cdot G_{i+1}/L_i\cdot L_{i+1}$,  endowed with the invariant complex structure $J$ given by the product of invariant complex structures  on each factor $M_i$, $i=1,3,\ldots,2k-1$.  \par 
We now consider an initial datum $(h_o^N,h_o^F)$, where $h_o^N$ is determined by a sequence $(A_1,\ldots,A_{2k})$. We can rearrange the indexes as in the proof of Proposition \ref{prop}, i.e. 
$A=A_1=\dots=A_p < A_j$ for all $j=p+1,\ldots,2k$. Note that  there is an involution $\s$ of the set $\{1,\ldots,2k\}$ so that for each index $1\leq l\leq 2k$ we have $JZ_l\in \Span\{Z_l,Z_{\s(l)}\}$. Hence
$$T_o(U^c\cdot o) = \gn_1\oplus\ldots\oplus \gn_p \oplus \Span\{Z_1,\ldots, Z_p,Z_{\s(1)},\ldots, Z_{\s(p)}\}.$$
Note that in this case the $S^c$-orbits are closed. When $t\to 2A$ the manifold $M$ collapses to a product of Hermitian symmetric spaces and a lower dimensional C-space $M'$ of Calabi-Eckmann type. Indeed, if we set $\mathcal I_1 :=  \{p+1,\ldots 2k\}\cap \s(\{p+1,\ldots,2k\})$ and $\mathcal I_2 := \{p+1,\ldots 2k\}\setminus \mathcal I_1$, then the manifold collapses to 
$$\left(\prod_{i\in \mathcal I_2} G_i/K_i\right) \times M',\qquad  M':= \prod_{i\in \mathcal I_1} G_i/L_i. $$


\begin{thebibliography}{ASM}
	
\bibitem{A}
	D.~Akhiezer.
	\newblock { Lie group actions in complex analysis}.
	\newblock Aspects of Mathematics, vol. E27 Vieweg (1995).	
\bibitem{AL}
	R.M.~Arroyo, R.A.~Lafuente.
	\newblock{\em The long-time behavior of the homogeneous pluriclosed flow}.
	\newblock arXiv:1712.02075v1 [math.DG] (2017).
\bibitem{B}
	J.~Boling.
	\newblock{\em Homogeneous solutions of pluriclosed flow on closed complex surfaces}. 
	\newblock{J. Geom. Anal.}, {\bf 26} (2016), 2130–-2154.
\bibitem{CE} 
	E.~Calabi, B.~Eckmann.
	\newblock{\em A new class of compact complex manifolds which are not algebraic}.
	\newblock {Ann. of Math.}, {\bf 58} (1953), 494--500.
\bibitem{FGV}
    A. Fino, G. Grantcharov, L. Vezzoni.
	\newblock {\em Astheno-K\"ahler and balanced structures on fibrations}.
	\newblock Int. Math. Res. Not. rnx337, https://doi.org/10.1093/imrn/rnx337.
\bibitem{G}
    P. Gauduchon.
    \newblock {\em Hermitian connections and Dirac operators}.
    \newblock Boll. Un. Mat. Ital. B(7), {\bf 11}(2) (1997), 257--288.
\bibitem{Gi}
    M. Gill.
    \newblock{\em Convergence of the parabolic complex Monge-Amp\`ere equation on compact Hermitian manifolds}.
    \newblock Comm. Anal. Geom. {\bf 19} (2011), 277--303.
\bibitem{KN}
	S. Kobayashi, K. Nomizu.
	\newblock Foundations of Differential Geometry.
	\newblock Interscience Tracts in Pure and Applied Mathematics, No. 15, Vol. II, John Wiley Sons, Inc.,
	          New York-London-Sidney 1969.
\bibitem{LPV}
    R.A.~Lafuente, M.~Pujia, L.~Vezzoni.
    \newblock {\em Hermitian Curvature flow on Lie groups and static invariant metrics}.
    \newblock arXiv:1807.00059 [math.DG] (2018).
\bibitem{M}
    N. Mok.
    \newblock {\em The uniformization theorem for compact K\"ahler manifolds of non-negative holomorphic bisectional curvature}.
    \newblock J. Diff. Geom. {\bf 27} (1988), 179--214.
\bibitem{P}
	F.~Podest\`a.
	\newblock {\em Homogeneous Hermitian manifolds and special metrics}.
	\newblock Transf. Groups {\bf 23}(4) (2018), 1129--1147.
\bibitem{ST2}
	J.~Streets, G.~Tian.
	\newblock{\em A parabolic flow of pluriclosed metrics}.
	\newblock Int. Math. Res. Not. 2010, no. 16, 3101--3133.
\bibitem{ST}
 	J.~Streets, G.~Tian. 
 	\newblock {\em Hermitian curvature flow}.
 	\newblock J. Eur. Math. Soc. {\bf 13} (2011), 601--634.
\bibitem{TW}
	V.~Tosatti, B.~Weinkove.
	\newblock {\em On the evolution of a Hermitian metric by its Chern-Ricci form}.
	\newblock J. Diff. Geom. {\bf 99} (2015), 125--163.
\bibitem{U}
   	Y.~Ustinovskiy.
   	\newblock {\em The Hermitian curvature flow on manifolds with non-negative Griffiths curvature}.
   	\newblock arXiv:1604.04813v2 [math.CV] (2016), to appear in Amer. J. Math.
\bibitem{U2}
	Y.~Ustinovskiy.
   	\newblock {\em Hermitian curvature flow on complex homogeneous manifolds}.
   	\newblock 	arXiv:1706.07023 [math.DG] (2017).
\bibitem{U3}
   	Y.~Ustinovskiy.
   	\newblock {\em Lie-algebraic curvature conditions preserved by the Hermitian curvature flow}.
   	\newblock 	arXiv:1710.06035 [math.DG] (2017).
\bibitem{W}
H.C.~Wang.
\newblock{\em Closed manifolds with homogenous complex structure}.
\newblock Amer. J. Math. {\bf 76} (1954), 1--32.


\end{thebibliography}
\end{document}